\documentclass[11pt,letterpaper,reqno]{amsart}
\usepackage{fullpage}
\usepackage[top=2cm, bottom=4.5cm, left=2.5cm, right=2.5cm]{geometry}
\usepackage{amsmath,amsthm,amsfonts,amssymb,amscd}
\usepackage{geometry,tikz,tikz-cd}
\usepackage{pgfplots}\pgfplotsset{compat=1.18}
\usepackage{empheq}
\usepackage{lipsum}
\usepackage{lastpage}
\usepackage{thmtools}
\usepackage{enumerate}
\usepackage[shortlabels]{enumitem}
\usepackage{fancyhdr}
\usepackage{mathrsfs}
\usepackage{xcolor}
\usepackage{graphicx}
\usepackage{listings}
\usepackage{bbm}
\usepackage{array}
\usepackage{hyperref}
\usepackage[makeroom]{cancel}
\usepackage{float}

\usepackage[font=footnotesize]{caption}
\usepackage{mwe}

\usepackage[utf8]{inputenc}
\usepackage{comment}

\usepackage{makecell}

\usepackage[all]{xy}
\xyoption{matrix}
\xyoption{arrow}

\usetikzlibrary{arrows,decorations.pathmorphing,backgrounds,positioning,fit,petri}

\tikzset{help lines/.style={step=#1cm,very thin, color=gray},
help lines/.default=.5} 
\tikzset{thick grid/.style={step=#1cm,thick, color=gray},
thick grid/.default=1} 

\mathtoolsset{showonlyrefs=true}
\numberwithin{figure}{section}
\numberwithin{table}{section}

\hypersetup{%
  colorlinks=true,
  linkcolor=blue,
  citecolor=blue,
  linkbordercolor={0 0 1} 
}

\newcommand{\lrabs}[1]{\left\lvert #1 \right\lvert}
\newcommand{\lrp}[1]{\left(#1\right)}
\newcommand{\lrb}[1]{\left[#1\right]}

\newcommand{\lrceil}[1]{\left\lceil #1 \right\rceil}

\allowdisplaybreaks

\theoremstyle{plain} 
\newtheorem{theorem}{Theorem}[section]

\newtheorem{lemma}[theorem]{Lemma}

\newtheorem{conjecture}[theorem]{Conjecture}
\newtheorem{proposition}[theorem]{Proposition}
\newtheorem{reformulation}[theorem]{Reformulation}
\theoremstyle{definition} 

\newtheorem{remark}[theorem]{Remark}

\numberwithin{equation}{section}

\usepackage{latexsym}
\usepackage{amssymb}

\newcommand{\ben}{\begin{equation}}
\newcommand{\een}{\end{equation}}

\makeatletter
\NewDocumentCommand{\sump}{e{_}}
 {%
  \DOTSB
  \mathop{\IfNoValueTF{#1}{\sump@{}}{\sump@{#1}}}%
  \nolimits
 }
\newcommand{\sump@}[1]{\mathpalette\sump@@{#1}}
\newcommand{\sump@@}[2]{%
  \ifx#1\displaystyle
    {\sump@display{#2}}%
  \else
    \sum@\nolimits'_{#2}%
  \fi
}
\newcommand{\sump@display}[1]{%
  \sbox\z@{$\m@th\displaystyle\sum@\nolimits'$}%
  \sbox\tw@{$\m@th\displaystyle\sum@\limits_{#1}$}%
  \sbox\@tempboxa{$\m@th\displaystyle'$}
  \mathop{\sum@\nolimits' \kern-\wd\@tempboxa}\limits_{#1}%
  \ifdim\wd\z@>\wd\tw@
    \kern\dimexpr\wd\z@-\wd\tw@\relax
  \fi
}
\makeatother

\DeclareMathOperator{\im}{Im}
\DeclareMathOperator{\re}{Re}

\setlist[enumerate]{leftmargin=*,widest=0}
\setlist[itemize]{leftmargin=*,widest=0}
\makeatletter
\def\subsection{\@startsection{subsection}{2}%
  \z@{.5\linespacing\@plus.7\linespacing}{.3\linespacing}%
  {\normalfont\bfseries}}

\def\subsubsection{\@startsection{subsubsection}{3}%
  \z@{.5\linespacing\@plus.7\linespacing}{.3\linespacing}%
  {\normalfont\bfseries}}
\makeatother

\title{Zeros of even and odd period polynomials}
\subjclass[2020]{11F11 and 11F67}

\keywords{Even period polynomial; odd period polynomials; zeros of polynomials; sign changes}

\begin{document}

\author[G. Ko]{Grace Ko}
\address[G. Ko]{Department of Mathematics\\
Vanderbilt University\\
Nashville, TN 37240}
\email{grace.s.ko@vanderbilt.edu}

\author[J. Mackenzie]{Jennifer Mackenzie}
\address[J. Mackenzie]{Department of Mathematics\\
Texas A\&M University\\
College Station, TX 77843}
\email{jennifer.mackenzie2@tamu.edu}

\author[E. Ross]{Erick Ross}
\address[E. Ross]{School of Mathematical and Statistical Sciences\\
Clemson University\\
Clemson, SC 29634}
\email{erickjohnross@gmail.com}

\author[H. Xue]{Hui Xue}
\address[H. Xue]{School of Mathematical and Statistical Sciences\\
Clemson University\\
Clemson, SC 29634}
\email{huixue@clemson.edu}

\begin{abstract}
    Let $f \in S_k(\Gamma_0(N))$ be a newform, and let $r_f^{\pm}(X)$ denote its corresponding even and odd period polynomials. For sufficiently large level and weight, we show that the zeros of $r_f^{\pm}(X)$ all lie on the circle $|X|  = \frac{1}{\sqrt N}$.
\end{abstract}

\maketitle

\section{Introduction}
Let $f\in S_k(\Gamma_0(N))$ be a newform of weight $k$ and level $N$. Write the Fourier expansion of $f$ as $f(z)=\sum_{n\ge1} a_f(n) q^n$. Let $L(f,s)=\sum_{n\ge1} a_f(n) n^{-s}$ be its associated $L$-function. Then associated to $f$ are the following three period polynomials:
\begin{align}
r_f(X) &:=-\frac{(k-2)!}{(2\pi i)^{k-1}} \sum_{0 \le n \le k-2} \frac{(2\pi iX)^n}{n!}L(f,k-n-1),\\
r_f^+(X) &:= \frac{i(k-2)!}{(2\pi i)^{k-1}} \sum_{\substack{0 \le n \le k-2 \\ 2 \mid n}}\frac{(2\pi iX)^n}{n!}L(f,k-n-1), \\
r_f^-(X) &:= -\frac{(k-2)!}{(2\pi i)^{k-1}} \sum_{\substack{1 \le n \le k-3 \\ 2 \nmid n}} \frac{(2\pi iX)^n}{n!}L(f,k-n-1).
\end{align}
These are called the (full) period polynomial of $f$, the even period polynomial of $f$, and the odd period polynomial of $f$, respectively. Under the action of the Fricke involution $W_N$, the newform $f$ satisfies 
\begin{align}
    f|_k W_N =\epsilon(f) f,
\end{align}
where $\epsilon(f)=\pm1$ is the called the Fricke sign of $f$. The Fricke sign $\epsilon(f)$ of $f$ is related to the sign $\hat \epsilon(f)$ of the functional equation of $L(f,s)$ via \cite[Section 5.10]{Diamond2005}
\begin{align}i^k \hat\epsilon(f)=\epsilon(f).
\end{align}
It follows from the functional equation of $L(f,s)$ that \cite[p.~2603]{jin-ma-ono-sound}
\begin{align}\label{eq:functional-equation-full}
    r_f(X) = - \epsilon(f)(\sqrt{N}X)^{k-2} r_f\lrp{\frac{-1}{NX}}.
\end{align} 
Taking even and odd parts of this identity then yields an identical functional equation for the even and odd period polynomials:
\begin{align}
 r_f^\pm(X) &= - \epsilon(f)(\sqrt{N}X)^{k-2}  \,r_f^\pm\lrp{\frac{-1}{NX}}.\label{eq:functional-equation-plusminus}
\end{align}

Now, for the full period polynomials $r_f(X)$, Jin-Ma-Ono-Soundararajan \cite{jin-ma-ono-sound} (and El Guindy-Wissam \cite{Guindy2014} in the case of level $N=1$) showed that all the zeros of $r_f(X)$ lie on the circle $|X|  = \frac{1}{\sqrt N}$. Jin-Ma-Ono-Soundararajan called this result the ``Riemann Hypothesis for period polynomials" since $|X| = \frac{1}{\sqrt N}$ is the circle of symmetry for the period polynomial functional equation.
In this paper, we show a similar result for the even and odd period polynomials $r_f^+(X)$ and $r_f^-(X)$.

\begin{theorem} \label{thm:main-thm-even}
    Define $\beta_1^+ := 2$, $\beta_N^+ := 1$ for $2 \le N \le 16$, and $\beta_N^+ = 0$ for $N \ge 17$. Then for newforms $f \in S_k(\Gamma_0(N))$ of weight $k \ge 12$, the even period polynomial $r_f^+(X)$ has at most $4\beta_N^+$ zeros off the circle $|X|=\frac{1}{\sqrt N}$.
\end{theorem}

\begin{theorem} \label{thm:main-thm-odd}
    Define $\beta_N^- := 1$ for $1 \le N \le 3$ and $\beta_N^- = 0$ for $N \ge 4$. Then for newforms $f \in S_k(\Gamma_0(N))$ of weight $k \ge 2\log_2 N$, the odd period polynomial $r_f^-(X)$ has at most $4\beta_N^-+1$ zeros off the circle $|X|=\frac{1}{\sqrt N}$.
\end{theorem}

We give four remarks here about Theorems \ref{thm:main-thm-even} and \ref{thm:main-thm-odd}.

\begin{remark}
    The odd period polynomial $r_f^-(X)$ trivially has a zero at $X=0$. Additionally, any non-zero zeros of $r_f^\pm(X)$ off the circle of symmetry come in quadruples $\pm X, \pm\frac{1}{NX}$ (by the functional equation \eqref{eq:functional-equation-plusminus} and the fact that $r_f^\pm(X)$ is an even/odd function).
    This is why, in Theorems \ref{thm:main-thm-even} and \ref{thm:main-thm-odd}, the number of zeros off the circle of symmetry are stated in the form $4\beta_N^+$ and $4\beta_N^-+1$. 
\end{remark}

\begin{remark} 
Several partial results in the direction of Theorems \ref{thm:main-thm-even} and \ref{thm:main-thm-odd} have been shown before. In the case of level $N=1$, Conrey-Farmer-Imamoglu \cite{Conrey2013} showed that all but $5$ zeros of $r_f^-(X)$ are located on the circle of symmetry. Then extending the method of \cite{Conrey2013}, Choi \cite{choi_zeros_odd} proved that for $N=2$ and $k\ge 96$, all but $5$ zeros of $r_f^-(X)$ are located on the circle of symmetry. Later, Choi \cite{choi_zeros_even} also showed that for sufficiently large $k$ (depending on $N$), all but finitely many of the zeros of $r_f^+(X)$ lie on the circle of symmetry. Very recently, another work appeared during the preparation of this manuscript; for sufficiently large $k$ (depending on $N$), Oh \cite{oh-new-paper} has shown that all but finitely many of the zeros of $r_f^-(X)$ lie on the circle of symmetry.

We note the improvements of Theorems \ref{thm:main-thm-even} and \ref{thm:main-thm-odd} over the previous results.
\begin{enumerate}
    \item
    Theorem \ref{thm:main-thm-even} improves the main result of \cite{choi_zeros_even} to all $k\ge12$. In particular; for any given weight $k \ge 12$, Theorem \ref{thm:main-thm-even} applies to all levels $N$, whereas \cite[Theorem 1.1]{choi_zeros_even} applied to only finitely many levels $N$.
    \item
    Theorem \ref{thm:main-thm-even} improves the values of $\beta_N^+$ from \cite{choi_zeros_even}. For $N \ge 17$ in particular, Theorem \ref{thm:main-thm-even} shows
    that all the zeros of $r_f^+(X)$ lie on the circle of symmetry, whereas \cite[Theorem 1.1]{choi_zeros_even} only showed this for newforms of Fricke sign $\epsilon(f)=+1$.
    \item 
    Theorem \ref{thm:main-thm-odd} improves the main results of \cite{Conrey2013} and \cite{choi_zeros_odd} to general level $N$.
    \item 
    Theorem \ref{thm:main-thm-odd} improves the main result of \cite{choi_zeros_odd} to general weight $k$.
    \item 
    Theorem \ref{thm:main-thm-odd} improves the main result of \cite{oh-new-paper} to the optimal values of $\beta_N^-$.
    \item 
    Theorems \ref{thm:main-thm-even} and \ref{thm:main-thm-odd} do not depend on the Fricke sign of $f$. 
    \item 
    This paper gives a unified approach to deal with both the even and odd period polynomials. All of the previous works in this area have only addressed one of these two cases (e.g. compare \cite{choi_zeros_odd} and \cite{choi_zeros_even}).
    \item 
    The approach used in this paper is much simpler than the previous approaches used in  \cite{Conrey2013,choi_zeros_odd,choi_zeros_even,oh-new-paper}. These works were all based off of the technical lemmas \cite[Lemmas 2.2 and 2.3]{Conrey2013}. In this paper, we use a new approach based off of \cite[Section 2]{InterlacingREU2024}, where we gave a simplified proof of \cite[Theorem 2.8]{Conrey2013}. This new approach follows more closely the strategy used for the full period polynomials $r_f(X)$ in Jin-Ma-Ono-Soundararajan \cite{jin-ma-ono-sound}.
    \item 
    The approach used in this paper gives very accurate estimates on the locations of the period polynomial zeros. This is useful because it paves the way for other results about the period polynomials. For example, in \cite{InterlacingREU2024}, we used this location information to prove an interlacing property for odd period polynomial zeros.
    \item 
    The approach used in this paper has the advantage that it more clearly shows \textit{why} there exist exceptional zeros for small levels. These exceptional zeros come from the argument jumps (corresponding to a winding number around the origin) in Lemmas \ref{lem:argument-even} and \ref{lem:argument-odd}.
\end{enumerate}
\end{remark}

\begin{remark} \label{remark:optimality}
    We make several observations concerning the optimality of Theorems \ref{thm:main-thm-even} and \ref{thm:main-thm-odd}.
    \begin{enumerate}
        \item 
        The lower bound $k \ge 12$ in Theorem \ref{thm:main-thm-even} is optimal. For each $6 \le k \le 10$, there exists a weight $k$ newform (of level $N=17$, for example) having a zero off the circle of symmetry \cite{ross-code}.
        \item 
        The values of $\beta_N^+$ in Theorem \ref{thm:main-thm-even} are optimal, except possibly for $\beta_{16}^+=1$. For each $1 \le N \le 15$, there exists a level $N$ newform (of weight $k=24$, for example) having $4\beta_N^+$ zeros off the circle of symmetry \cite{ross-code}.
        \item
        If Conjecture \ref{conj:N16-endpoint-zero} holds, then the value of $\beta_{16}^+$ in Theorem \ref{thm:main-thm-even} could be updated to $0$, which would then necessarily be optimal.
        \item 
        The lower bound $k \ge 2 \log_2 N$ in Theorem \ref{thm:main-thm-odd} is probably not optimal. Computational evidence seems to suggest that the optimal lower bound is 
        $k \ge 12$ \cite{ross-code}.
        \item 
        The values of $\beta_N^-$ in Theorem \ref{thm:main-thm-odd} are optimal. For each $1 \le N \le 3$, there exists a level $N$ newform (of weight $k=24$, for example) having $4\beta_N^-+1$ zeros off the circle of symmetry \cite{ross-code}.
    \end{enumerate}
\end{remark}

\begin{remark}
    We point out several questions that arise from Theorems \ref{thm:main-thm-even} and \ref{thm:main-thm-odd}. 
    We will not address any of these questions in this paper.
    However, they could potentially make for interesting future projects.
    \begin{enumerate}
        \item
        The most obvious open question is Conjecture \ref{conj:N16-endpoint-zero}. As mentioned in Remark \ref{remark:optimality}, if one could prove this conjecture, then the value of $\beta_{16}^+$ in Theorem \ref{thm:main-thm-even} could be updated to $0$. We have already shown the case of $N=16,\ \epsilon(f)=+1$ (see Case 6 in the proof of Theorem \ref{thm:main-thm-even}). And the case of $N=16,\ \epsilon(f)=-1$ would follow from Conjecture \ref{conj:N16-endpoint-zero} and Lemma \ref{lem:N16-double-zeros-endpoints} (see Case 7 in the proof of Theorem \ref{thm:main-thm-even}).
        \item 
        For $1 \le N \le 15$, one might ask where the $4\beta_N^+$ exceptional zeros from Theorem \ref{thm:main-thm-even} are located.
        In the case of $2 \le N \le 15$; when they exist, the $4\beta_N^+ = 4(1)$ exceptional zeros seem to be located at, or near to $\pm \frac{1}{4},\pm \frac{4}{N}$. 
        In the case of $N=1$; when they exist, the $4\beta_1^+ = 4(2)$ exceptional zeros seem to be located at, or near to $\pm \frac{1}{4},\pm 4$ and $\pm \frac{3}{4},\pm \frac{4}{3}$ \cite{ross-code}.
        It would be an interesting result to prove such a pattern.
        \item
        For $1 \le N \le 3$, one could similarly ask where the
        $4\beta_N^-$ non-zero exceptional zeros from Theorem \ref{thm:main-thm-odd} are located.
        In the case of $N=1,2$; \cite{Conrey2013} and \cite{choi_zeros_odd} showed that $r_f^-(X)$ always have exceptional zeros located at $\pm \frac{1}{2},\pm \frac{2}{N}$.
        In the case of $N=3$; when they exist, the $4\beta_N^- = 4(1)$ non-zero exceptional zeros seem to be located at, or near to $\pm \frac{1}{2},\pm \frac{2}{N}$ \cite{ross-code}.
        \item 
        In Theorems \ref{thm:main-thm-even} and \ref{thm:main-thm-odd}, the best possible lower bound on weight is $k \ge 12$ (see Remark \ref{remark:optimality}). However, it would still be interesting to investigate what happens for smaller weights $k$. Although Theorems \ref{thm:main-thm-even} and \ref{thm:main-thm-odd} would not hold in general, many of the newforms of weight $k < 12$ still happen to have all their zeros on the circle of symmetry. 
    \end{enumerate}
\end{remark}

Finally, we give an outline of the paper.
In Section \ref{sec:endpoint-zeros}, we first take note of certain ``endpoint zeros" that are trivially guaranteed to exist by the functional equation \eqref{eq:functional-equation-plusminus}. Then in Section \ref{sec:reformulation}, we reformulate the question about zeros of period polynomials into a question about zeros of certain real-valued functions.
Then throughout the rest of the paper, our approach will be to lower-bound the number of zeros of these real-valued functions by counting how many times they change sign. 
To achieve this, in Section \ref{sec:approximation-functions} we introduce certain approximation functions for the above real-valued functions. 
Then in Section \ref{sec:error-terms}, we obtain explicit error bounds on these approximations.
Finally, in Sections \ref{sec:proof-main-thm-even} and \ref{sec:proof-main-thm-odd} we complete the proofs of Theorems \ref{thm:main-thm-even} and \ref{thm:main-thm-odd}.

\section{Endpoint zeros} \label{sec:endpoint-zeros}
We begin by noting that in certain cases, $r_f^{+}(X)$ and $r_f^-(X)$ are trivially guaranteed to have ``endpoint zeros" at $X=\pm\frac{1}{\sqrt N}$ (the terminology ``endpoint zeros" will become clear at the end of the next section). In particular, the functional equation \eqref{eq:functional-equation-plusminus} immediately implies the following lemma.
\begin{lemma}\label{lem:zeros-endpoints}
Let $f\in S_k(\Gamma_0(N))$ be a newform.
    \begin{enumerate}
        \item If $\epsilon(f)=+1$, then $r_f^+(X)$ has a zero at $X=\pm\frac{1}{\sqrt{N}}$.
        \item If $\epsilon(f)=-1$, then  $r_f^-(X)$ has a zero at $X=\pm\frac{1}{\sqrt{N}}$.
    \end{enumerate}
\end{lemma}

We let $\delta_{f}^+$ and $\delta_f^-$ denote the number of these trivial endpoint zeros (with the pair $X=\pm \frac{1}{\sqrt N}$ counted only once). In particular, we define
\begin{align}
    \label{eqn:delta-f-def-plus}
    \delta_f^+ := 
    \begin{cases}
        1 & \text{if }\ \ \epsilon(f)=+1, \\
        0 & \text{if }\ \ \epsilon(f)=-1, 
    \end{cases} \\
    \label{eqn:delta-f-def-minus}
    \delta_f^- := 
    \begin{cases}
        1 & \text{if }\ \ \epsilon(f)=-1, \\
        0 & \text{if }\ \ \epsilon(f)=+1.
    \end{cases}
\end{align}

We will not need this fact for the paper, but we also conjecture here that $r_f^+(X)$ in fact has endpoint zeros for \textit{every} newform $f$ of level $N=16$. We have verified the following conjecture computationally for all weights $k \le 100$ \cite{ross-code}.
\begin{conjecture} \label{conj:N16-endpoint-zero}
    Let $f \in S_k(\Gamma_0(16))$ be a newform of level $N=16$. Then $r_f^+\lrp{\pm \frac14} = 0$.
\end{conjecture}
In this conjecture, only the case of $\epsilon(f)=-1$ remains open; Lemma \ref{lem:zeros-endpoints} already implies the case of $\epsilon(f)=+1$.
Also,  when $\epsilon(f)=-1$, this conjecture would furthermore imply that $X=\pm \frac14$ is a double zero of $r_f^+(X)$ (see Lemma \ref{lem:N16-double-zeros-endpoints} below). Hence, if Conjecture \ref{conj:N16-endpoint-zero} holds, then the value of $\delta_f^+$ could be updated to $2$ when $N=16,\ \epsilon(f)=-1$.

\begin{lemma}\label{lem:N16-double-zeros-endpoints}
Let $f\in S_k(\Gamma_0(16))$ be a newform of level $N=16$. If $\epsilon(f)=-1$ and $r^+_f\lrp{\pm \frac14}=0$, then $X=\pm \frac{1}{4}$ is a double zero of $r_f^+(X)$.
\end{lemma}
\begin{proof}
    Suppose that $\epsilon(f)=-1$ and $r^+_f\lrp{\pm \frac14}=0$. In this case, the functional equation \eqref{eq:functional-equation-plusminus} becomes
   \begin{align} 
       r_f^+(X) = (4X)^{k-2} r_f^+\lrp{\frac{1}{16X}}.
   \end{align}
   Then taking the derivative of this identity at $X =\pm \frac{1}{4}$ yields 
   \begin{align} 
       {r_f^+}'\lrp{\pm\frac14} = \pm4(k-2)\, r_f^+\lrp{\pm\frac{1}{4}} -{r_f^+}'\lrp{\pm\frac14} = -{r_f^+}'\lrp{\pm\frac14}.
   \end{align}
   This implies that ${r_f^+}'\lrp{\pm\frac14} = 0$, proving the desired result.
\end{proof}

\section{A reformulation} \label{sec:reformulation}

In this section,  
we use the period polynomial functional equations to rewrite the period polynomials $r_f^{\pm}(X)$ in terms of certain closely related polynomials $q_f^\pm(X)$. We will then reformulate the desired questions about period polynomial zeros into questions about the behavior of these $q_f^\pm(X)$.

Here, and throughout the rest of the paper, we will use the notation $m := \frac{k-2}{2}$. We then normalize the period polynomials as follows (see \cite[p.~917]{choi_zeros_even}, \cite[p.~774]{choi_zeros_odd}):
\begin{align}
    \hat{r}_f^+(X):=&\frac{r_f^+(X)}{(-1)^{m} (2\pi)^{-2m-1}(2m)!}=\sum_{n=0}^{m}\frac{(-1)^n(2\pi X)^{2n}}{(2n)!} L(f, 2m+1-2n), 
    \label{eqn:r-hat-plus} \\
    \hat{r}_f^-(X):=&\frac{-r_f^-(X)}{(-1)^{m}(2\pi)^{-2m-1} (2m)!}=\sum_{n=0}^{m-1}\frac{(-1)^n (2\pi X)^{2n+1}}{(2n+1)!} L(f,2m+1-(2n+1)).
    \label{eqn:r-hat-minus}
\end{align}

Then by \cite[Lemma 3.3]{choi_zeros_even} and \cite[Lemma 3.4]{choi_zeros_odd}, we can write
\begin{align}\label{eqn:rf-qf}
    \hat{r}_f^\pm(X) = q^\pm_f(X) \mp \epsilon(f)\lrp{\sqrt{N}X}^{2m} q^\pm_f\lrp{\frac{1}{NX}},
\end{align}
where $q^+_f(X)$ and $q^-_f(X)$ are the polynomials (with real coefficients) given by 
\begin{align}
    \label{eq:qf_def_even}
    q^+_f(X) &= 
    \sideset{}{'}\sum_{0 \le n \le m/2}  \frac{(-1)^n(2\pi X)^{2n}}{(2n)!}L(f,2m+1-2n), \\
    \label{eq:qf_def_odd}
    q_f^-(X) &= 
    \sideset{}{'}\sum_{0 \le n \le (m-1)/2}\frac{(-1)^n(2\pi X)^{2n+1}}{(2n+1)!}L(f,2m+1-(2n+1)).
\end{align} 
As usual, the $\sum'$ here means that if the last term ($n=m/2$ and $n=(m-1)/2$, respectively) appears in the summation, then it is counted with weight $1/2$.

Next, we parametrize points $X$ lying on the circle of symmetry $|X| = \frac{1}{\sqrt N}$ via $X=\frac{e^{i\theta}}{\sqrt{N}}$ for $\theta \in [0,2\pi)$. Then under this parametrization, 
the identity \eqref{eqn:rf-qf} becomes:
\begin{align}
    \label{eqn:Xcircle-rhat-q-even}
    e^{-i m\theta} \hat{r}_f^+\lrp{ \frac{e^{i \theta}}{\sqrt N}}&=
   \begin{cases} 
        2i \im\lrb{e^{-im\theta}q_f^+\lrp{ \frac{e^{i \theta}}{\sqrt N}}} & 
        \quad\text{if $\epsilon(f)=+1$,} \\
        \phantom{i}2 \re\lrb{e^{-im\theta}q_f^+\lrp{ \frac{e^{i \theta}}{\sqrt N}}} & 
        \quad\text{if $\epsilon(f)=-1$,}
   \end{cases}\\
   \label{eqn:Xcircle-rhat-q-odd}
   e^{-i m\theta} \hat{r}_f^-\lrp{ \frac{e^{i \theta}}{\sqrt N}}&=
   \begin{cases}
      2i \im\lrb{e^{-im\theta}q_f^-\lrp{ \frac{e^{i \theta}}{\sqrt N}}} & 
      \quad\text{if $\epsilon(f)=-1$}, \\
      \phantom{i}2\re\lrb{e^{-im\theta}q_f^-\lrp{ \frac{e^{i \theta}}{\sqrt N}}} & 
      \quad\text{if $\epsilon(f)=+1$}.
   \end{cases}
\end{align}

Finally, we connect all of this setup to the goal of the paper: showing that the period polynomial zeros lie on the circle of symmetry. By \eqref{eqn:Xcircle-rhat-q-even} and \eqref{eqn:Xcircle-rhat-q-odd}, the zeros of $\hat{r}_f^{\pm}(X)$ on the circle of symmetry are precisely the zeros of the imaginary/real part of $e^{-im\theta}q_f^\pm\lrp{ \frac{e^{i \theta}}{\sqrt N}}$ for $\theta \in [0,2\pi)$. 

Thus since $\hat{r}_f^{+}(X)$ is an even polynomial of degree $2m$, it has $m$ pairs of zeros (paired via negation; i.e. rotation by $\pi$ in the complex plane), and so $m-\delta_f^+$ non-endpoint pairs of zeros.
Hence, to show Theorem \ref{thm:main-thm-even} (that $\hat{r}_f^{+}(X)$ has at most $2\beta_N^+$ pairs of zeros off the circle of symmetry), it suffices to show the following reformulation. Note that in this reformulation, the $\delta_f^+$ term is accounting for the endpoint zeros already known to exist at $\theta=0$ (by \eqref{eqn:delta-f-def-plus}).

{
\renewcommand{\thereformulation}{of Theorem \ref{thm:main-thm-even}}
\begin{reformulation}
    For newforms $f \in S_k(\Gamma_0(N))$ of weight $k \ge 12$, the imaginary/real part (depending on whether $\epsilon(f)=+1$ or $\epsilon(f)=-1$) of $e^{-im\theta}q_f^+\lrp{ \frac{e^{i \theta}}{\sqrt N}}$ has at least $m-\delta_f^+ - 2\beta_N^+$ distinct zeros for $\theta \in (0,\pi)$. 
\end{reformulation}
\addtocounter{reformulation}{-1}
}

Similarly, since $\hat{r}_f^{-}(X)$ is an odd polynomial of degree $2m-1$, it has one zero at $0$, plus $m-1$ pairs of zeros. 
Hence, to show Theorem \ref{thm:main-thm-odd} (that $\hat{r}_f^{-}(X)$ has at most $2\beta_N^-$ pairs of non-zero zeros off the circle of symmetry), it suffices to show the following reformulation.

{
\renewcommand{\thereformulation}{of Theorem \ref{thm:main-thm-odd}}
\begin{reformulation}
    For newforms $f \in S_k(\Gamma_0(N))$ of weight $k \ge 2 \log_2 N$, the imaginary/real part (depending on whether $\epsilon(f)=-1$ or $\epsilon(f)=+1$) of $e^{-im\theta}q_f^-\lrp{ \frac{e^{i \theta}}{\sqrt N}}$ has at least $m-1-\delta_f^- - 2\beta_N^-$ distinct zeros for $\theta \in (0,\pi)$.
\end{reformulation}
\addtocounter{reformulation}{-1}
}

Throughout the rest of the paper, we will use this reformulated version of the problem. In particular, we now focus on studying the imaginary/real part of $e^{-im\theta}q_f^\pm\lrp{ \frac{e^{i \theta}}{\sqrt N}}$.

\section{Approximation functions} \label{sec:approximation-functions}

To understand the behavior of $e^{-im\theta}q_f^\pm\lrp{ \frac{e^{i \theta}}{\sqrt N}}$, we will approximate the functions $q_f^+\lrp{ \frac{e^{i \theta}}{\sqrt N}}$ and $q_f^-\lrp{ \frac{e^{i \theta}}{\sqrt N}}$ in Section \ref{sec:error-terms} by:
\begin{align} 
    \label{eqn:approx-qf-even}
    q_f^+\lrp{ \frac{e^{i \theta}}{\sqrt N}} 
    &\approx \cos\lrp{\frac{2\pi e^{i\theta}}{\sqrt{N}}}, \\
    \label{eqn:approx-qf-odd}
    q_f^-\lrp{ \frac{e^{i \theta}}{\sqrt N}} 
    &\approx \sin\lrp{\frac{2\pi e^{i\theta}}{\sqrt{N}}}.
\end{align}

In this section, we collect certain properties of these approximation functions $\cos\lrp{\frac{2\pi e^{i\theta}}{\sqrt{N}}}$ and $\sin\lrp{\frac{2\pi e^{i\theta}}{\sqrt{N}}}$. In particular, Subsection \ref{subsec:argument-function} addresses their arguments, and Subsection \ref{subsec:radius-function} addresses their radii.

\subsection{Arguments of the approximation functions} \label{subsec:argument-function}

Observe that
\begin{align} 
    &\cos\lrp{\frac{2\pi e^{i\theta}}{\sqrt{N}}}\\
    =& \cos\lrp{\frac{2\pi}{\sqrt{N}}\cos\theta+i \frac{2\pi}{\sqrt{N}}\sin\theta}\\
    =&\cos\lrp{\frac{2\pi}{\sqrt{N}}\cos\theta}
    \cosh \lrp{\frac{2\pi}{\sqrt{N}} \sin \theta}
    -i
    \sin\lrp{\frac{2\pi}{\sqrt{N}}\cos\theta}
    \sinh \lrp{\frac{2\pi}{\sqrt{N}} \sin \theta}.
    \label{eqn:coordinate-even-approx}
\end{align}
Hence the tangent of its argument is given by
\begin{align}
    -\frac{
        \sin\lrp{\frac{2\pi}{\sqrt{N}}\cos\theta}
        \sinh \lrp{\frac{2\pi}{\sqrt{N}} \sin \theta}
    }{
        \cos\lrp{\frac{2\pi}{\sqrt{N}}\cos\theta}
        \cosh \lrp{\frac{2\pi}{\sqrt{N}} \sin \theta}
    } =
    -\tan\lrp{\frac{2\pi}{\sqrt{N}}\cos\theta}
    \tanh \lrp{\frac{2\pi}{\sqrt{N}} \sin \theta}. 
\end{align}
Similarly, observe that
\begin{align} 
        &\sin\lrp{2\pi \frac{e^{i\theta}}{\sqrt{N}}}\\ = &\sin\lrp{\frac{2\pi}{\sqrt{N}} \cos\theta + i\frac{2\pi}{\sqrt{N}}\sin\theta}\\
         =&\sin\lrp{\frac{2\pi}{\sqrt{N}}\cos\theta}
         \cosh \lrp{\frac{2\pi}{\sqrt{N}} \sin \theta}
         + i
         \cos\lrp{\frac{2\pi}{\sqrt{N}}\cos\theta}
         \sinh \lrp{\frac{2\pi}{\sqrt{N}} \sin \theta}
         .
         \label{eqn:coordinate-odd-approx}
    \end{align}
Hence the tangent of its argument is given by
\begin{align}
    \frac{
        \cos\lrp{\frac{2\pi}{\sqrt{N}}\cos\theta}
        \sinh \lrp{\frac{2\pi}{\sqrt{N}} \sin \theta}
    }{
        \sin\lrp{\frac{2\pi}{\sqrt{N}}\cos\theta}
        \cosh \lrp{\frac{2\pi}{\sqrt{N}} \sin \theta}
    } =
    \cot\lrp{\frac{2\pi}{\sqrt{N}}\cos\theta}
    \tanh \lrp{\frac{2\pi}{\sqrt{N}} \sin \theta}. 
\end{align}

Now, define  
\begin{align}
    \hat{a}_N^+(\theta) &:= \arctan \lrp{
        -\tan\lrp{\frac{2\pi}{\sqrt{N}}\cos\theta}
        \tanh \lrp{\frac{2\pi}{\sqrt{N}} \sin \theta}
    }, \label{eq:a-hat-N-even}
    \\
    \hat{a}^-_N(\theta) &:=\arctan \lrp{
        \cot\lrp{\frac{2\pi}{\sqrt{N}}\cos\theta}
        \tanh \lrp{\frac{2\pi}{\sqrt{N}} \sin \theta}
    }. \label{eq:a-hat-N-odd}
\end{align}
Recall that, by convention, $\arctan(\cdot)$ is defined to always output values between $-\pi/2$ and $\pi/2$. This convention means that $\hat{a}^{\pm}_N(\theta)$ here will not necessarily be a continuous function of $\theta$. However, one can force it to be continuous by carefully adding on certain multiples of $\pi$. In particular, we now define $a_N^\pm(\theta)$, which is a continuous argument function for the approximation functions $\cos\lrp{ \frac{2\pi e^{i\theta}}{\sqrt{N}}}$ and $\sin\lrp{\frac{e^{2\pi i\theta}}{\sqrt{N}}}$.

The claims of the following two lemmas can easily be seen by basic calculus computations or by a graphing calculator \cite{ross-code}; see Figures \ref{fig:arguments-even} and \ref{fig:arguments-odd} for illustrations for $N=3$. 

\begin{lemma} \label{lem:argument-even}
For $\theta\in[0,\pi]$, the argument $a_N^+(\theta)$ of $\cos\lrp{\frac{2\pi e^{i\theta}}{\sqrt{N}}}$ is continuous and is given by the following.
\begin{enumerate}
    \item If $N\ge17$, then $a_N^+(\theta)=\hat{a}_N^+(\theta)$. Furthermore, $a_N^+(0)=a_N^+(\pi)=0$.
        
    \item If $N=16$, then $a_{16}^+(\theta)=\hat{a}_{16}^+(\theta)$. Furthermore, $a_{16}^+(0)=-\frac{\pi}{2}$ and $a_{16}^+(\pi)=\frac{\pi}{2}$.
        
    \item If $2 \le N\le15$, then 
    \begin{equation}
    a_{N}^+(\theta)=
    \begin{cases}
        \hat{a}_{N}^+(\theta) +\pi &\text{for } 0 \le \theta < \arccos\lrp{{\frac{\sqrt{N}}{4}}},\\
        \hat{a}_{N}^+(\theta) + 2\pi &\text{for } \arccos\lrp{{\frac{\sqrt{N}}{4}}} \le \theta < \arccos\lrp{{\frac{-\sqrt{N}}{4}}},\\
        \hat{a}_{N}^+(\theta) + 3\pi &\text{for } \arccos\lrp{{\frac{-\sqrt{N}}{4}}} \le \theta \le \pi.
    \end{cases}
    \end{equation}
    Furthermore, $a_N^+(0)=\pi$ and $a_N^+(\pi)=3\pi$.
    
    \item If $N=1$, then
    \begin{equation}
    a_{1}^+(\theta)=
    \begin{cases}
        \hat{a}_{1}^+(\theta) &\text{for } 0 \le \theta < \arccos\lrp{{\frac{3}{4}}},\\
        \hat{a}_{1}^+(\theta) + \pi &\text{for } \arccos\lrp{{\frac{3}{4}}} \le \theta < \arccos\lrp{{\frac{1}{4}}},\\
        \hat{a}_{1}^+(\theta) +2\pi &\text{for } \arccos\lrp{{\frac{1}{4}}} \le \theta < \arccos\lrp{{\frac{-1}{4}}},\\
        \hat{a}_{1}^+ (\theta)+ 3\pi &\text{for } \arccos\lrp{{\frac{-1}{4}}} \le \theta < \arccos\lrp{{\frac{-3}{4}}},\\
        \hat{a}_{1}^+(\theta) + 4\pi &\text{for } \arccos\lrp{{\frac{-3}{4}}} \le \theta \le  \pi.
    \end{cases}
    \end{equation}
    Furthermore, $a_1^+(0)=0$ and $a_1^+(\pi)=4\pi$.
\end{enumerate}
\end{lemma}

\begin{figure}[H]
\centering
\begin{tikzpicture}[scale=0.70]
\begin{axis}[
    axis 
    lines=left,
    xmin=0,
    ymin=-2,
    ymax=15,
]
\addplot [
    domain=0+0.001:rad(acos(sqrt(3)/4))-0.001,
    samples=200, 
    color=red,
    thick,
]
{rad(atan(
((-sin(deg(2*pi*cos(deg(x))/sqrt(3))))*(e^(2*pi*sin(deg(x))/sqrt(3)) - e^(-2*pi*sin(deg(x))/sqrt(3))))/((cos(deg(2*pi*cos(deg(x))/sqrt(3))))*(e^(2*pi*sin(deg(x))/sqrt(3)) + e^(-2*pi*sin(deg(x))/sqrt(3))))
))};
\addplot [
    domain=rad(acos(sqrt(3)/4)):rad(acos(-sqrt(3)/4))-0.001, 
    samples=200, 
    color=blue,
    thick,
]
{rad(atan(
((-sin(deg(2*pi*cos(deg(x))/sqrt(3))))*(e^(2*pi*sin(deg(x))/sqrt(3)) - e^(-2*pi*sin(deg(x))/sqrt(3))))/((cos(deg(2*pi*cos(deg(x))/sqrt(3))))*(e^(2*pi*sin(deg(x))/sqrt(3)) + e^(-2*pi*sin(deg(x))/sqrt(3))))
))};
\addplot [
    domain=rad(acos(-sqrt(3)/4))+0.001:pi, 
    samples=200, 
    color=red,
    thick,
]
{rad(atan(
((-sin(deg(2*pi*cos(deg(x))/sqrt(3))))*(e^(2*pi*sin(deg(x))/sqrt(3)) - e^(-2*pi*sin(deg(x))/sqrt(3))))/((cos(deg(2*pi*cos(deg(x))/sqrt(3))))*(e^(2*pi*sin(deg(x))/sqrt(3)) + e^(-2*pi*sin(deg(x))/sqrt(3))))
))};
\end{axis}
\end{tikzpicture}
\qquad
\begin{tikzpicture}[scale=0.70]
\begin{axis}[
    axis 
    lines=left,
    xmin=0,
    ymin=0,
    ymax=15,
]
\addplot [
    domain=0+0.001:rad(acos(sqrt(3)/4))-0.001,
    samples=200, 
    color=red,
    thick,
]
{rad(atan(
((-sin(deg(2*pi*cos(deg(x))/sqrt(3))))*(e^(2*pi*sin(deg(x))/sqrt(3)) - e^(-2*pi*sin(deg(x))/sqrt(3))))/((cos(deg(2*pi*cos(deg(x))/sqrt(3))))*(e^(2*pi*sin(deg(x))/sqrt(3)) + e^(-2*pi*sin(deg(x))/sqrt(3))))
)) + pi};
\addplot [
    domain=rad(acos(sqrt(3)/4)):rad(acos(-sqrt(3)/4))-0.001, 
    samples=200, 
    color=blue,
    thick,
]
{rad(atan(
((-sin(deg(2*pi*cos(deg(x))/sqrt(3))))*(e^(2*pi*sin(deg(x))/sqrt(3)) - e^(-2*pi*sin(deg(x))/sqrt(3))))/((cos(deg(2*pi*cos(deg(x))/sqrt(3))))*(e^(2*pi*sin(deg(x))/sqrt(3)) + e^(-2*pi*sin(deg(x))/sqrt(3))))
)) + 2*pi};
\addplot [
    domain=rad(acos(-sqrt(3)/4))+0.001:pi, 
    samples=200, 
    color=red,
    thick,
]
{rad(atan(
((-sin(deg(2*pi*cos(deg(x))/sqrt(3))))*(e^(2*pi*sin(deg(x))/sqrt(3)) - e^(-2*pi*sin(deg(x))/sqrt(3))))/((cos(deg(2*pi*cos(deg(x))/sqrt(3))))*(e^(2*pi*sin(deg(x))/sqrt(3)) + e^(-2*pi*sin(deg(x))/sqrt(3))))
)) + 3*pi};
\end{axis}
\end{tikzpicture}
    \caption{The left plot is a graph of $\hat{a}_3^+(\theta)$ over the interval $[0,\pi]$. The right plot is a graph of $a_3^+(\theta)$ over the same interval. Note the discontinuities that lend themselves to the definition of the argument function $a_3^+(\theta)$ in Lemma \ref{lem:argument-even}.}
    \label{fig:arguments-even}
\end{figure}

\begin{lemma} \label{lem:argument-odd}
For $\theta\in[0, \pi]$, the argument $a_N^-(\theta)$ of $\sin\lrp{\frac{2\pi e^{i\theta}}{\sqrt{N}}}$ is continuous and is given by the following.
\begin{enumerate}
    \item If $N\ge5$, then
    \begin{equation}\label{eq:alpha-N>=5}
    a^-_N(\theta)=
    \begin{cases} 
        \hat{a}_N^-(\theta)
        &\text{for } 0 \le \theta < \frac{\pi}{2},\\
        \hat{a}_N^-(\theta)+ \pi 
        &\text{for } \frac{\pi}{2} \leq \theta \le \pi.
    \end{cases}
    \end{equation}
    Furthermore, $a_N^-(0)=0$ and $a_N^-(\pi)=\pi$.
    \item If $N=4$, then
    \begin{equation}
    a^-_4(\theta)=
    \begin{cases}  
        \hat{a}_4^-(\theta) 
        &\text{for } 0 \le \theta < \frac{\pi}{2},\\
        \hat{a}_4^-(\theta)+ \pi 
        &\text{for } \frac{\pi}{2} \leq \theta \le \pi.
    \end{cases}
    \end{equation}
    Furthermore, $a_4^-(0)=-\frac{\pi}{2}$ and $a_4^-(\pi)=\frac{3\pi}{2}$. 
    \item If $2 \le N\le 3$, then 
    \begin{equation}
    a_N^-(\theta) = 
    \begin{cases}
        \hat{a}_N^-(\theta) + \pi &\text{for } 0 \le \theta < \arccos\lrp{\frac{\sqrt{N}}{2}},\\
        \hat{a}_N^-(\theta) + 2\pi &\text{for } \arccos\lrp{\frac{\sqrt{N}}{2}} \le \theta < \frac{\pi}{2},\\
        \hat{a}_N^-(\theta) + 3\pi &\text{for } \frac{\pi}{2} \le \theta < \arccos\lrp{\frac{-\sqrt{N}}{2}},\\
        \hat{a}_N^-(\theta) + 4\pi &\text{for } \arccos\lrp{\frac{-\sqrt{N}}{2}} \le \theta \le \pi.
    \end{cases}
    \end{equation}
    Furthermore, $a_{N}^-(0)=\pi$ and $a_N^-(\pi)=4\pi$.
\end{enumerate}
\end{lemma}
We note that Lemma \ref{lem:argument-odd} omits the case of level $N=1$ since the case of odd period polynomials of level one was already dealt with in \cite{Conrey2013}.

\begin{figure}[H]
\centering
\begin{tikzpicture}[scale=0.70]
\begin{axis}[
    axis
    lines=left,
    xmin=0,
    ymin=-2,
    ymax=15,
]
\addplot [
    domain=0+0.001:rad(acos(sqrt(3)/2))-0.001, 
    samples=200, 
    color=red,
    thick,
]
{rad(atan(
((cos(deg(2*pi*cos(deg(x))/sqrt(3))))*(e^(2*pi*sin(deg(x))/sqrt(3)) - e^(-2*pi*sin(deg(x))/sqrt(3))))/((sin(deg(2*pi*cos(deg(x))/sqrt(3))))*(e^(2*pi*sin(deg(x))/sqrt(3)) + e^(-2*pi*sin(deg(x))/sqrt(3))))
))};
\addplot [
    domain=rad(acos(sqrt(3)/2)):pi/2, 
    samples=200, 
    color=blue,
    thick,
]
{rad(atan(
((cos(deg(2*pi*cos(deg(x))/sqrt(3))))*(e^(2*pi*sin(deg(x))/sqrt(3)) - e^(-2*pi*sin(deg(x))/sqrt(3))))/((sin(deg(2*pi*cos(deg(x))/sqrt(3))))*(e^(2*pi*sin(deg(x))/sqrt(3)) + e^(-2*pi*sin(deg(x))/sqrt(3))))
))};
\addplot [
    domain=pi/2 +0.001:rad(acos(-sqrt(3)/2)), 
    samples=200, 
    color=red,
    thick,
]
{rad(atan(
((cos(deg(2*pi*cos(deg(x))/sqrt(3))))*(e^(2*pi*sin(deg(x))/sqrt(3)) - e^(-2*pi*sin(deg(x))/sqrt(3))))/((sin(deg(2*pi*cos(deg(x))/sqrt(3))))*(e^(2*pi*sin(deg(x))/sqrt(3)) + e^(-2*pi*sin(deg(x))/sqrt(3))))
))};
\addplot [
    domain=rad(acos(-sqrt(3)/2))+0.001:pi, 
    samples=200, 
    color=blue,
    thick,
]
{rad(atan(
((cos(deg(2*pi*cos(deg(x))/sqrt(3))))*(e^(2*pi*sin(deg(x))/sqrt(3)) - e^(-2*pi*sin(deg(x))/sqrt(3))))/((sin(deg(2*pi*cos(deg(x))/sqrt(3))))*(e^(2*pi*sin(deg(x))/sqrt(3)) + e^(-2*pi*sin(deg(x))/sqrt(3))))
))};
\end{axis}
\end{tikzpicture}
\qquad
\begin{tikzpicture}[scale=0.70]
\begin{axis}[
    axis
    lines=left,
    xmin=0,
    ymin=0,
    ymax=15,
]
\addplot [
    domain=0+0.001:rad(acos(sqrt(3)/2))-0.001, 
    samples=200, 
    color=red,
    thick,
]
{rad(atan(
((cos(deg(2*pi*cos(deg(x))/sqrt(3))))*(e^(2*pi*sin(deg(x))/sqrt(3)) - e^(-2*pi*sin(deg(x))/sqrt(3))))/((sin(deg(2*pi*cos(deg(x))/sqrt(3))))*(e^(2*pi*sin(deg(x))/sqrt(3)) + e^(-2*pi*sin(deg(x))/sqrt(3))))
)) + pi};
\addplot [
    domain=rad(acos(sqrt(3)/2)):pi/2, 
    samples=200, 
    color=blue,
    thick,
]
{rad(atan(
((cos(deg(2*pi*cos(deg(x))/sqrt(3))))*(e^(2*pi*sin(deg(x))/sqrt(3)) - e^(-2*pi*sin(deg(x))/sqrt(3))))/((sin(deg(2*pi*cos(deg(x))/sqrt(3))))*(e^(2*pi*sin(deg(x))/sqrt(3)) + e^(-2*pi*sin(deg(x))/sqrt(3))))
)) + 2*pi};
\addplot [
    domain=pi/2 +0.001:rad(acos(-sqrt(3)/2)), 
    samples=200, 
    color=red,
    thick,
]
{rad(atan(
((cos(deg(2*pi*cos(deg(x))/sqrt(3))))*(e^(2*pi*sin(deg(x))/sqrt(3)) - e^(-2*pi*sin(deg(x))/sqrt(3))))/((sin(deg(2*pi*cos(deg(x))/sqrt(3))))*(e^(2*pi*sin(deg(x))/sqrt(3)) + e^(-2*pi*sin(deg(x))/sqrt(3))))
)) + 3*pi};
\addplot [
    domain=rad(acos(-sqrt(3)/2))+0.001:pi, 
    samples=200, 
    color=blue,
    thick,
]
{rad(atan(
((cos(deg(2*pi*cos(deg(x))/sqrt(3))))*(e^(2*pi*sin(deg(x))/sqrt(3)) - e^(-2*pi*sin(deg(x))/sqrt(3))))/((sin(deg(2*pi*cos(deg(x))/sqrt(3))))*(e^(2*pi*sin(deg(x))/sqrt(3)) + e^(-2*pi*sin(deg(x))/sqrt(3))))
)) + 4*pi};
\end{axis}
\end{tikzpicture}
    \caption{The left plot is a graph of $\hat{a}_3^-(\theta)$ over the interval $[0,\pi]$. The right plot is a graph of $a_3^-(\theta)$ over the same interval. Note the discontinuities that lend themselves to the definition of the argument function $a_3^-(\theta)$ in Lemma \ref{lem:argument-odd}.}
    \label{fig:arguments-odd}
\end{figure}

\subsection{Radii of the approximation functions} \label{subsec:radius-function}
Next, we consider the radii of the approximation functions $\cos\lrp{\frac{2\pi e^{i\theta}}{\sqrt{N}}}$ and $\sin\lrp{\frac{2\pi e^{i\theta}}{\sqrt{N}}}$. By \eqref{eqn:coordinate-even-approx} and \eqref{eqn:coordinate-odd-approx}, these radii are given by
\begin{align} \label{eq:radius-even}
    r_{N}^+(\theta) &:= \lrabs{\cos\lrp{\frac{2\pi e^{i\theta}}{\sqrt{N}}}} 
    = 
    \sqrt{
        \frac{1}{2} \cosh \lrp{ \frac{4\pi\sin\theta}{\sqrt{N}}}
        + \frac{1}{2} \cos\lrp{\frac{4\pi\cos\theta}{\sqrt{N}}}
    }, \\
    \label{eq:radius-odd}
    r_{N}^-(\theta) &:=
    \lrabs{\sin\lrp{\frac{2\pi e^{i\theta}}{\sqrt{N}}}}
    =
    \sqrt{
        \frac{1}{2} 
        \cosh \lrp{ \frac{4\pi\sin\theta}{\sqrt{N}}}
        - \frac{1}{2} 
        \cos\lrp{\frac{4\pi\cos\theta}{\sqrt{N}}}
    }. 
\end{align}

We then note some basic properties of the radius functions $r^{\pm}_N(\theta)$. The following lemma can be easily verified via direct calculus computation or by graphing calculator \cite{ross-code}.

\begin{lemma} \label{lemma:radius-increasing}
    For all $N \ge 1$, $r_N^\pm(\theta)$ is symmetric about $\frac{\pi}{2}$ over the interval $[0,\pi]$, and is increasing over $[0,\pi/2]$. In particular, for all $\theta \in [0,\pi]$, we have
    $r_N^+(\theta) \ge r_N^+(0) = \lrabs{\cos\lrp{\frac{2\pi}{\sqrt N}}}$ 
    and 
    $r_N^-(\theta) \ge r_N^-(0) = \lrabs{\sin\lrp{\frac{2\pi}{\sqrt N}}}$.
\end{lemma}

\section{Bounds on error terms} \label{sec:error-terms}

In this section, we show that $\cos\lrp{\frac{2\pi e^{i\theta}}{\sqrt{N}}}$ and $\sin\lrp{\frac{2\pi e^{i\theta}}{\sqrt{N}}}$ actually are good approximation functions for $q_f^+\lrp{ \frac{e^{i \theta}}{\sqrt N}}$ and $q_f^-\lrp{ \frac{e^{i \theta}}{\sqrt N}}$ (from \eqref{eqn:approx-qf-even}, \eqref{eqn:approx-qf-odd}).
In particular, we obtain explicit bounds on the differences between $q_f^{\pm}\lrp{ \frac{e^{i \theta}}{\sqrt N}}$ and their approximation functions. Many of these arguments are modifications of Choi's work in \cite{choi_zeros_odd,choi_zeros_even}.

To accomplish this, we will first need the following $L$-function value estimates.
\begin{lemma}\label{lem:bounds-L-values}
    Let $k\ge12$ and  $f\in S_k(\Gamma_0(N))$ be a newform. 
    \begin{itemize}
        \item[(1)]  For any integer $\sigma\ge\frac{3k}{4}$, we have 
        \begin{align}
            |L(f,\sigma)-1| \le 2.71 \cdot2^{-k/4}
        \end{align}
        \item[(2)] For any integer $\sigma\ge \frac{k}{2}$, we have
        \begin{align}
            |L(f,\sigma)| 
            &\le 4\sqrt{k}(1+\log k)+\frac{4}{e^{\frac{\pi}{\sqrt{N}}}-1} e^{\frac{-\pi k}{\sqrt{N}}} 2^{k/2}.
        \end{align}
    \end{itemize}
\end{lemma}
\begin{proof}
    (2) was already shown in \cite[Lemma 2.1]{choi_zeros_even}. Here, we give a proof of (1).
    
    If $\sigma\ge \frac{3k}{4}$ and $d(n)$ is the divisor function, then by Deligne's theorem  
    \begin{align}
        |L(f,\sigma)-1| \le \sum_{n=2}^{\infty}\frac{d(n)}{n^{\sigma-(k-1)/2}} \le \sum_{n=2}^{\infty} \frac{d(n)}{n^{k/4+1/2}} 
        =\zeta\lrp{\frac{k+2}{4}}^2-1.
    \end{align}
    Furthermore,
    \begin{align}
        \zeta\lrp{\frac{k+2}{4}}^2-1 &=\lrp{\zeta\lrp{\frac{k+2}{4}}+1}\lrp{\zeta\lrp{\frac{k+2}{4}}-1} \\
        &\le \lrp{\zeta\lrp{\frac{14}{4}}+1}\lrp{\zeta\lrp{\frac{k+2}{4}}-1},
    \end{align}
    where
    \begin{align}
        \zeta\lrp{\frac{k+2}{4}}-1
        &= 2^{-\frac{k+2}{4}}+\sum_{n\ge3} n^{-\frac{k+2}{4}} \\
        &\le 2^{-\frac{k+2}{4}}+\int_{2}^{\infty} u^{-\frac{k+2}{4}}\,du \\
        &= 2^{-\frac{k+2}{4}}+\frac{1}{\frac{k-2}{4}} 2^{-\frac{k-2}{4}}\\
        &\le \sqrt2 \lrp{\frac12 + \frac{4}{10}} 2^{-\frac{k}{4}}. 
    \end{align}
    This verifies (1), since $ \lrp{\zeta\lrp{\frac{14}{4}}+1} \sqrt2 \lrp{\frac12 + \frac{4}{10}} \le 2.71$.
\end{proof}


\subsection{Even period polynomials} \label{subsec:bounds-even}


Denote the difference between $q_f^+\lrp{ \frac{e^{i \theta}}{\sqrt N}}$ and its approximation function $\cos\lrp{\frac{2\pi e^{i\theta}}{\sqrt{N}}}$ by
\begin{align} \label{eq:Ef-def-even}
    E_f^+(\theta) 
    &:= q_f^+\lrp{ \frac{e^{i \theta}}{\sqrt N}} -\cos\lrp{\frac{2\pi e^{i\theta}}{\sqrt{N}}}
    =S_1^+(\theta) + S_2^+(\theta) - S_3^+(\theta),
\end{align}
where 
\begin{align}
    S_1^+(\theta) &:= \sum_{j < m/4}\frac{(-1)^j
    \lrp{ \frac{2 \pi e^{i \theta}}{\sqrt N}}^{2j}}{(2j)!}(L(f,2m+1-2j)-1),\\
    S_2^+(\theta) &:= 
    \sump_{m/4 \le j\le\frac{m}{2}}\frac{
        (-1)^j \lrp{ \frac{2 \pi  e^{i \theta}}{\sqrt N}}^{2j}
    }{(2j)!}
    L(f,2m+1-2j),\\
    S_3^+(\theta) &:= \sum_{j \ge m/4}\frac{(-1)^j \lrp{ \frac{2 \pi  e^{i \theta}}{\sqrt N}}^{2j}}{(2j)!}.
\end{align}

By Lemma \ref{lem:bounds-L-values}, this then yields the bounds
\begin{align}
    |S_1^+(\theta)| &\le  
    2.71 \cdot2^{-k/4}
    \sum_{j<m/4} \frac{\lrp{\frac{2\pi}{\sqrt N}}^{2j}}{(2j)!}, 
    \label{eqn:temp-S1-plus-bound} \\
    |S_2^+(\theta)| &\le  \lrp{4\sqrt{k}(1+\log k)+\frac{4}{e^{\frac{\pi}{\sqrt{N}}}-1} e^{\frac{-\pi k}{\sqrt{N}}} 2^{k/2}}  \sump_{m/4 \le j \le m/2} \frac{\lrp{\frac{2\pi}{\sqrt N}}^{2j}}{(2j)!}, 
    \label{eqn:temp-S2-plus-bound} \\
    |S_3^+(\theta)| &\le  \sum_{j \ge m/4} \frac{\lrp{\frac{2\pi}{\sqrt N}}^{2j}}{(2j)!}.
    \label{eqn:temp-S3-plus-bound}
\end{align}
After some inspection, one can see that each of these bounds tends to $0$ exponentially quickly in $k$ (for specific bounding details, see \eqref{eqn:SNinfk-exponentially-small} in the proof of Proposition \ref{prop:error<radius-odd} below). In particular, these bounds allow us to prove the following proposition.

\begin{proposition} \label{prop:error<radius-even}
    Let $f \in S_k(\Gamma_0(N))$ be a newform of weight $k \ge k_N^+$ (as given below).
        \begin{center}
        \begin{tabular}{|c||c|c|c|c|c|c|c|c|}
            \hline
            $N$      & 1 & 2 & 3 & 4-5 & 6-18 & 19-29 & 30-103 & $\ge$104 \\
            \hline
            $k_N^+$ & 76 & 60 & 52 & 44 & 36 & 28 & 20 & 12 \\
            \hline
        \end{tabular}
    \end{center}
    
    \begin{enumerate}
        \item 
        If $N \ne 16$, then  $|E_f^+(\theta)| < r_N^+(\theta)$ for all $\theta \in [0, \pi]$. 
        \item 
        If $N=16$, then $|E_f^+(\theta)| < \frac{1}{2m}$ for all $\theta \in [0,\pi]$.
    \end{enumerate}
\end{proposition}
\begin{proof}
    Define $S_{N,M,k}^+$ as
    \begin{align}
        S_{N,M,k}^+ := \,& 
        2.71 \cdot2^{-k/4}
        \sum_{j<m/4} \frac{\lrp{\frac{2\pi}{\sqrt N}}^{2j}}{(2j)!}
        +
        \sum_{j \ge m/4} \frac{\lrp{\frac{2\pi}{\sqrt N}}^{2j}}{(2j)!} \\
        \,&+
        \lrp{4\sqrt{k}(1+\log k)+\frac{4}{e^{\frac{\pi}{\sqrt{N}}}-1} e^{\frac{-\pi k}{\sqrt{M}}} 2^{k/2}}  \sump_{m/4 \le j \le m/2} \frac{\lrp{\frac{2\pi}{\sqrt N}}^{2j}}{(2j)!}.
    \end{align}
    Observe that $S_{N,M,k}^+$ is increasing in $M$, and decreasing in $N$ (using the fact that $\frac{4}{e^x - 1} \cdot (2x)^4$ is increasing over $x\in[0, \pi]$). 
    
    Now, for each fixed pair 
    \begin{align}
        (N_0,M_0) \in \{
        &(1,1), (2,2), (3,3), (4,4), (5,5), (6,6), (7,7), (8,8), 
        (9,9), (10,10), (11,11), (12,12), \\
        & (13,13), (14,14), (15,15), 
        (17,18), (19,29),(30,39),(40,103), (104,104), (105,109),\\
        &(110,119), (120,169), (170,249), (250,399), (400,699), (700,1599), (1600,\infty)\},
    \end{align}
    consider the levels $N$ in the range $N_0 \le N \le M_0$.  
    For such $N$, we have the upper bound $\lrabs{E_f^+(\theta)} \le S_{N,N,k}^+ \le S_{N_0,M_0,k}^+$ by \eqref{eqn:temp-S1-plus-bound}, \eqref{eqn:temp-S2-plus-bound}, \eqref{eqn:temp-S3-plus-bound}.
    Similarly, we have the lower bound $r_N^+(\theta) \ge \lrabs{\cos\lrp{\frac{2\pi}{\sqrt N}}} \ge \lrabs{\cos\lrp{\frac{2\pi}{\sqrt N_0}}}$ by Lemma \ref{lemma:radius-increasing} (and using the fact that $N \mapsto \lrabs{\cos\lrp{\frac{2\pi}{\sqrt N}}}$ is increasing for $N \ge 16$).
    
    Hence to show (1), it suffices to show
    that $S_{N_0,M_0,k}^+ < \lrabs{\cos\lrp{\frac{2\pi}{\sqrt N_0}}}$ for all $k \ge k_{N_0}^+$. 
    This fact can be easily verified for each pair $(N_0,M_0)$ via direct calculus computation or by graphing calculator \cite{ross-code}. 

    Similarly, to show (2), it suffices to show that $m \cdot S_{16,16,k}^+ < \frac12$ for all $k \ge k_{16}^+ = 36$. Again, this fact can be easily verified via direct calculus computation or by graphing calculator \cite{ross-code}. 
\end{proof}


\subsection{Odd period polynomials} \label{subsec:bounds-odd}

Denote the difference between $q_f^-\lrp{ \frac{e^{i \theta}}{\sqrt N}}$ and its approximation function $\sin\lrp{\frac{2\pi e^{i\theta}}{\sqrt{N}}}$ by
\begin{align} \label{eq:Ef-def-odd}
    E_f^-(\theta) 
    &:= q_f^-\lrp{ \frac{e^{i \theta}}{\sqrt N}} -\sin\lrp{\frac{2\pi e^{i\theta}}{\sqrt{N}}}
    =S_1^-(\theta) + S_2^-(\theta) - S_3^-(\theta),
\end{align}
where 
\begin{align}
    S_1^-(\theta) &:= \sum_{j < (m-2)/4}\frac{(-1)^j
    \lrp{ \frac{2 \pi e^{i \theta}}{\sqrt N}}^{2j+1}}{(2j+1)!}(L(f,w-2j)-1),\\
    S_2^-(\theta) &:= 
    \sump_{(m-2)/4 \le j\le (m-1)/2}\frac{
        (-1)^j \lrp{ \frac{2 \pi  e^{i \theta}}{\sqrt N}}^{2j+1}
    }{(2j+1)!}
    L(f,w-2j),\\
    S_3^-(\theta) &:= \sum_{j \ge (m-2)/4}\frac{(-1)^j \lrp{ \frac{2 \pi  e^{i \theta}}{\sqrt N}}^{2j+1}}{(2j+1)!}.
\end{align}

By Lemma \ref{lem:bounds-L-values}, we then obtain the bounds
\begin{align}
    |S_1^-(\theta)| &\le  
    2.71 \cdot2^{-k/4}
    \sum_{j<(m-2)/4} \frac{\lrp{\frac{2\pi}{\sqrt N}}^{2j+1}}{(2j+1)!}, 
    \label{eqn:temp-S1-minus-bound} \\
    |S_2^-(\theta)| &\le  \lrp{4\sqrt{k}(1+\log k)+\frac{4}{e^{\frac{\pi}{\sqrt{N}}}-1} e^{\frac{-\pi k}{\sqrt{N}}} 2^{k/2}}  \sump_{(m-2)/4 \le j \le (m-1)/2} \frac{\lrp{\frac{2\pi}{\sqrt N}}^{2j+1}}{(2j+1)!}, 
    \label{eqn:temp-S2-minus-bound} \\
    |S_3^-(\theta)| &\le  \sum_{j \ge (m-2)/4} \frac{\lrp{\frac{2\pi}{\sqrt N}}^{2j+1}}{(2j+1)!}.
    \label{eqn:temp-S3-minus-bound}
\end{align}

These bounds then allow us to prove the following proposition.
\begin{proposition} \label{prop:error<radius-odd}
    Let $f \in S_k(\Gamma_0(N))$ be a newform of level $N \ge 2$ and weight $k \ge k_N^-$ (as given below).
    \begin{center}
        \begin{tabular}{|c|c|c|c|c|c|c|}
            \hline
            $N$     & 2-4 & 5-7 & 8-13 & 14-39 & 40-2499 & $\ge$2500 \\
            \hline
            $k_N^-$ & 56 & 40 & 32 & 24 & 16 &  $2 \log_2 N$ \\
            \hline
        \end{tabular}
    \end{center}
    \begin{enumerate}
        \item 
        If $N \ne 4$, then $|E_f^-(\theta)| < r_N^-(\theta)$ 
        for all $\theta \in [0, \pi]$.
        \item 
        If $N = 4$, then $|E_f^-(\theta)| < \frac{1}{2m}$ 
        for all $\theta \in [0,\pi]$.
    \end{enumerate} 
\end{proposition}
\begin{proof} 

Define $S_{N,M,k}^-$ as
    \begin{align}
        S_{N,M,k}^- := \,& 
        2.71 \cdot2^{-k/4}
        \sum_{j<(m-2)/4} \frac{\lrp{\frac{2\pi}{\sqrt N}}^{2j+1}}{(2j+1)!}
        +
        \sum_{j \ge (m-2)/4} \frac{\lrp{\frac{2\pi}{\sqrt N}}^{2j+1}}{(2j+1)!} 
        \label{eqn:S-N-M-k-minus}
        \\
        \,&+
        \lrp{4\sqrt{k}(1+\log k)+\frac{4}{e^{\frac{\pi}{\sqrt{N}}}-1} e^{\frac{-\pi k}{\sqrt{M}}} 2^{k/2}}  \sump_{(m-2)/4 \le j \le (m-1)/2} \frac{\lrp{\frac{2\pi}{\sqrt N}}^{2j+1}}{(2j+1)!}.
    \end{align}
    Just like before, $S_{N,M,k}^-$ is increasing in $M$, and decreasing in $N$.
    
    Now, for each fixed pair 
    \begin{align}
        (N_0,M_0) \in \{
        &(2,2), (3,3), (5,5), (6,6), (7,7), (8,8), 
        (9,9), (10,10), (11,11), (12,12), (13,13), \\
        &(14,14), (15,15), 
        (16,39), (40,49), (50,109), (110,279), (280,699), (700,2499)\},
    \end{align}
    consider the levels $N$ in the range $N_0 \le N \le M_0$.  
    For such $N$, we have the upper bound $\lrabs{E_f^-(\theta)} \le S_{N,N,k}^- \le S_{N_0,M_0,k}^-$ by \eqref{eqn:temp-S1-minus-bound}, \eqref{eqn:temp-S2-minus-bound}, \eqref{eqn:temp-S3-minus-bound}.
    Similarly, we have the lower bound $r_N^-(\theta) \ge \lrabs{\sin\lrp{\frac{2\pi}{\sqrt N}}} \ge \lrabs{\sin\lrp{\frac{2\pi}{\sqrt M_0}}}$ by Lemma \ref{lemma:radius-increasing} (and using the fact that $N \mapsto \lrabs{\sin\lrp{\frac{2\pi}{\sqrt N}}}$ is decreasing for $N \ge 16$).
    
    Hence to show (1) for $N \le 2499$, it suffices to show
    that $S_{N_0,M_0,k}^- < \lrabs{\sin\lrp{\frac{2\pi}{\sqrt M_0}}}$ for all $k \ge k_{N_0}^-$. 
    This fact can be easily verified for each pair $(N_0,M_0)$ via direct calculus computation or by graphing calculator \cite{ross-code}. 

    Similarly, to show (2), it suffices to show that $m \cdot S_{4,4,k}^- < \frac12$ for all $k \ge k_{4}^- = 56$. Again, this fact can be easily verified via direct calculus computation or by graphing calculator \cite{ross-code}.

    Finally, we show (1) for $N \ge 2500$. In this case, it suffices to show that $S_{N,\infty,k}^- < \lrabs{\sin\lrp{\frac{2\pi}{\sqrt N}}}$ for all $k \ge k_{N}^-$. We prove this inequality via more detailed estimates of $S_{N,\infty,k}^-$.

    Observe that each of the summations in \eqref{eqn:S-N-M-k-minus} is made up of terms for the exponential power series. And by the Taylor remainder theorem for $\exp(x)$, we have the bound
    \begin{align}
        \sum_{j \ge n} \frac{x^j}{j!} \le \frac{\exp(x) x^n}{n!} \qquad \text{for $x \ge 0$}.
    \end{align}
    This means that
    \begin{align}
        \sum_{j<(m-2)/4} \frac{\lrp{\frac{2\pi}{\sqrt N}}^{2j+1}}{(2j+1)!} 
        &\le \exp\lrp{\frac{2\pi}{\sqrt N}}
        \label{eqn:temp4}
        \\
        \text{and} \quad\sum_{j \ge (m-2)/4} \frac{\lrp{\frac{2\pi}{\sqrt N}}^{2j+1}}{(2j+1)!} 
        &\le 
        \frac{\exp \lrp{\frac{2\pi}{\sqrt N}} \, \lrp{\frac{2\pi}{\sqrt N}}^{\lrceil{m/2}}}{\lrceil{m/2} !}.
        \label{eqn:temp5}
    \end{align}
    Additionally, observe that 
    \begin{align} \label{eqn:temp6}
        4\sqrt{k}(1+\log k)+\frac{4}{e^{\frac{\pi}{\sqrt{N}}}-1} 2^{k/2}
        \le \lrp{\frac{4}{\pi} \sqrt N + 1}
        2^{k/2}.
    \end{align}
    This follows immediately from the facts that $4\sqrt{k}(1+\log k) \le 2^{k/2}$ (since $k \ge 24$)
    and $\frac{1}{e^x-1} \le \frac1x$ (taken at $x = \frac{\pi}{\sqrt N}$). 

    Then applying \eqref{eqn:temp4}, \eqref{eqn:temp5}, \eqref{eqn:temp6} to \eqref{eqn:S-N-M-k-minus}, we obtain
    \begin{align}
        S_{N,\infty,k}^- 
        &\le  
             2.71 \cdot 2^{-k/4} 
             \exp\lrp{\frac{2\pi}{\sqrt N}}
             + 
             \frac{\exp \lrp{\frac{2\pi}{\sqrt N}} \, \lrp{\frac{2\pi}{\sqrt N}}^{\lrceil{m/2}}}{\lrceil{m/2} !}
             \\
             & \qquad+
             \lrp{\frac{4}{\pi} \sqrt N + 1} 2^{k/2}
             \frac{\exp \lrp{\frac{2\pi}{\sqrt N}} \, \lrp{\frac{2\pi}{\sqrt N}}^{\lrceil{m/2}}}{\lrceil{m/2} !} 
         \\
        &= 2^{-k/4} \exp\lrp{\frac{2\pi}{\sqrt N}} \lrb{
             2.71 
             + 
            \lrp{\frac{4}{\pi} \sqrt N + 2} 2^{3k/4}
             \frac{ \lrp{\frac{2\pi}{\sqrt N}}^{\lrceil{m/2}}}{\lrceil{m/2} !} 
        } \\
        &= 2^{-k/4} \exp\lrp{\frac{2\pi}{\sqrt N}} \lrb{
             2.71
             + 
            \lrp{8 + \frac{4\pi}{\sqrt N}} 2^{3k/4}
             \frac{ \lrp{\frac{2\pi}{\sqrt N}}^{\lrceil{m/2}-1}}{\lrceil{m/2} !} 
        }  \\
        &\le 2^{-k/4} \exp\lrp{\frac{2\pi}{50}} \lrb{
             2.71
             + 
            \lrp{8 + \frac{4\pi}{50}} 2^{3k/4}
             \frac{ \lrp{\frac{2\pi}{50}}^{\lrceil{m/2}-1}}{\lrceil{m/2} !} 
        }. \qquad \text{(since $N \ge 2500$)}
    \end{align}
    Now, it is straightforward to verify that $2^{3k/4}  \lrp{\frac{2\pi}{50}}^{\lrceil{m/2}-1} / \lrceil{m/2} ! \le 0.04$ (since $k \ge 24$).  
    This then means that
    \begin{align}
        S_{N,\infty,k}^- 
        &\le 2^{-k/4}\exp\lrp{\frac{2\pi}{50}} \lrb{
             2.71  
             + 
            \lrp{8+\frac{4\pi}{50}}\cdot 0.04  
        } \le 3.45 \cdot 2^{-k/4} \label{eqn:SNinfk-exponentially-small}
        \\
        &\le 3.45 \cdot 2^{- (2\log_2 N)/4} = \frac{3.45}{\sqrt N} 
        < \sin\lrp{\frac{2\pi}{\sqrt N}}, 
    \end{align}
    proving the desired result.
\end{proof}

\section{Proof of Theorem \ref{thm:main-thm-even}} 
\label{sec:proof-main-thm-even}

In this section, we prove Theorem \ref{thm:main-thm-even}. Recall (by the reformulation of Section \ref{sec:reformulation}) that it suffices to show that the imaginary/real part of $e^{-im\theta}q_f^+\lrp{ \frac{e^{i \theta}}{\sqrt N}}$ has at least $m-\delta_f^+ - 2\beta_N^+$ distinct zeros over $\theta \in (0,\pi)$. 
\begin{proof} 
    First, observe that there are only finitely many pairs $(N,k)$ with $12 \le k < k_N^+$ (see Proposition \ref{prop:error<radius-even}). Hence to prove the desired result for $k < k_N^+$, we just checked these finitely many cases by computer \cite{ross-code}. For rest of the proof, we will then assume that $k \ge k_N^+$, so that we can apply Proposition \ref{prop:error<radius-even}.
    
    \textbf{Case 1: $N \ge 17,\ \epsilon(f)=+1$.\\}
    In this case, we need to show that $\im\lrb{e^{-im\theta} q_f^+\lrp{\frac{e^{i\theta}}{\sqrt N}}}$ has at least $m-1$ distinct zeros over $\theta \in (0,\pi)$. Writing $q_f^+\lrp{\frac{e^{i\theta}}{\sqrt N}}$ in terms of its approximation function (see \ref{eq:Ef-def-even}), we obtain
    \begin{align}
        \im\lrb{e^{-im\theta} q_f^+\lrp{\frac{e^{i\theta}}{\sqrt N}}}
        &= \im\lrb{e^{-im\theta}  \cos\lrp{\frac{2\pi e^{i\theta}}{\sqrt N}} 
        + e^{-im\theta} E_f^+(\theta)
        } \\
        &= \sin \lrp{a_N^+(\theta)-m \theta} r_N^+(\theta) + \im\lrb{e^{-im\theta} E_f^+(\theta)}. \label{eqn:temp.2.1}
    \end{align}
    (Note that $a_N^+(\theta)-m \theta$ here is the argument of $e^{-im\theta}  \cos\lrp{\frac{2\pi e^{i\theta}}{\sqrt N}}$ and $r_N^+(\theta)$ is its radius.) Then recall that $\lrabs{E_f^+(\theta)} < r_N^+(\theta)$ (by Proposition \ref{prop:error<radius-even}). Hence \eqref{eqn:temp.2.1} means that 
    $\im\lrb{e^{-im\theta} q_f^+\lrp{\frac{e^{i\theta}}{\sqrt N}}}$ will have the same sign as $\sin\lrp{a_N^+(\theta)-m \theta}$
    whenever $\sin\lrp{a_N^+(\theta)-m \theta} = \pm 1$. 

    Now, $a_N^+(\theta)-m \theta$ changes continuously from $0$ to $-m\pi$ over $\theta \in [0,\pi]$ (by Lemma \ref{lem:argument-even}), and so $\sin\lrp{a_N^+(\theta)-m \theta} = \pm 1$ (alternatingly) $m$ times over the same interval. Hence $\im\lrb{e^{-im\theta} q_f^+\lrp{\frac{e^{i\theta}}{\sqrt N}}}$ has $m-1$ sign changes over $\theta \in (0,\pi)$, yielding $m-1$ distinct zeros in $(0,\pi)$, as desired.

    \textbf{Case 2: $N \ge 17,\ \epsilon(f)=-1$.\\}
    In this case, we need to show that $\re\lrb{e^{-im\theta} q_f^+\lrp{\frac{e^{i\theta}}{\sqrt N}}}$ has $m$ at least distinct zeros over $\theta \in (0,\pi)$. Just like before, we have
    \begin{align}
        \re\lrb{e^{-im\theta} q_f^+\lrp{\frac{e^{i\theta}}{\sqrt N}}}
        &= \re\lrb{e^{-im\theta}  \cos\lrp{\frac{2\pi e^{i\theta}}{\sqrt N}} 
        + e^{-im\theta} E_f^+(\theta)
        } \\
        &= \cos \lrp{a_N^+(\theta)-m \theta} r_N^+(\theta) + \re\lrb{e^{-im\theta} E_f^+(\theta)}. \label{eqn:temp.2.2}
    \end{align}
    Hence 
    $\re\lrb{e^{-im\theta} q_f^+\lrp{\frac{e^{i\theta}}{\sqrt N}}}$ will have the same sign as $\cos\lrp{a_N^+(\theta)-m \theta}$
    whenever $\cos\lrp{a_N^+(\theta)-m \theta} = \pm 1$. 

    Now, $a_N^+(\theta)-m \theta$ changes continuously from $0$ to $-m\pi$ over $\theta \in [0,\pi]$ (by Lemma \ref{lem:argument-even}), and so $\cos\lrp{a_N^+(\theta)-m \theta} = \pm 1$ (alternatingly) $m+1$ times over the same interval. Hence $\re\lrb{e^{-im\theta} q_f^+\lrp{\frac{e^{i\theta}}{\sqrt N}}}$ has $m$ sign changes over $\theta \in (0,\pi)$, yielding $m$ distinct zeros in $(0,\pi)$, as desired.
        
    \textbf{Case 3: $2 \le N \le 15,\ \epsilon(f)=+1$.\\}
    In this case, we need to show that $\im\lrb{e^{-im\theta} q_f^+\lrp{\frac{e^{i\theta}}{\sqrt N}}}$ has at least $m-3$ distinct zeros over $\theta \in (0,\pi)$.

    Here, $a_N^+(\theta)-m \theta$ changes continuously from $\pi$ to $(-m+3)\pi$ over $\theta \in [0,\pi]$ (by Lemma \ref{lem:argument-even}), and so $\sin\lrp{a_N^+(\theta)-m \theta} = \pm 1$ (alternatingly) $m-2$ times over the same interval. Hence $\im\lrb{e^{-im\theta} q_f^+\lrp{\frac{e^{i\theta}}{\sqrt N}}}$ has $m-3$ sign changes over $\theta \in (0,\pi)$, yielding $m-3$ distinct zeros in $(0,\pi)$, as desired.

    \textbf{Case 4: $2 \le N \le 15,\ \epsilon(f)=-1$.\\}
    In this case, we need to show that $\re\lrb{e^{-im\theta} q_f^+\lrp{\frac{e^{i\theta}}{\sqrt N}}}$ has at least $m-2$ distinct zeros over $\theta \in (0,\pi)$.

    Here, $a_N^+(\theta)-m \theta$ changes continuously from $\pi$ to $(-m+3)\pi$ over $\theta \in [0,\pi]$ (by Lemma \ref{lem:argument-even}), and so $\cos\lrp{a_N^+(\theta)-m \theta} = \pm 1$ (alternatingly) $m-1$ times over the same interval. Hence $\re\lrb{e^{-im\theta} q_f^+\lrp{\frac{e^{i\theta}}{\sqrt N}}}$ has $m-2$ sign changes over $\theta \in (0,\pi)$, yielding $m-2$ distinct zeros in $(0,\pi)$, as desired.

    \textbf{Case 5: $N=1$.\\}
    Note that $\epsilon(f)$ is necessarily equal to $+1$.
    In this case, we need to show that $\im\lrb{e^{-im\theta} q_f^+\lrp{\frac{e^{i\theta}}{\sqrt N}}}$ has at least $m-5$ distinct zeros over $\theta \in (0,\pi)$.

    Here, $a_N^+(\theta)-m \theta$ changes continuously from $0$ to $(-m+4)\pi$ over $\theta \in [0,\pi]$ (by Lemma \ref{lem:argument-even}), and so $\sin\lrp{a_N^+(\theta)-m \theta} = \pm 1$ (alternatingly) $m-4$ times over the same interval. Hence $\im\lrb{e^{-im\theta} q_f^+\lrp{\frac{e^{i\theta}}{\sqrt N}}}$ has $m-5$ sign changes over $\theta \in (0,\pi)$, yielding $m-5$ distinct zeros in $(0,\pi)$, as desired.

    \textbf{Case 6: $N=16,\ \epsilon(f)=+1$.\\}
    In this case, we need to show that $\im\lrb{e^{-im\theta} q_f^+\lrp{\frac{e^{i\theta}}{\sqrt 16}}}$ has at least $m-3$ distinct zeros over $\theta \in (0,\pi)$.

    Here, $a_{16}^+(\theta)-m\theta$ changes continuously from $-\frac{\pi}{2}$ to $\frac{\pi}{2}-m\pi$ over $\theta \in [0,\pi]$ (by Lemma \ref{lem:argument-even}), and so $\sin\lrp{a_{16}^+(\theta)-m \theta} = \pm 1$ (alternatingly) $m$ times over the same interval. 
    
    However, unlike all the previous cases, Proposition \ref{prop:error<radius-even} here only yields the bound $\lrabs{E_f^+(\theta)} < \frac{1}{2m}$. This bound turns out to still be sufficient to prove what we need.
    Let $\theta_1 \in (0,\pi)$ be such that $a_{16}^+(\theta_1)-m\theta_1 = -\frac{3\pi}{4}$. (Note that this also means that $a_{16}^+(\pi-\theta_1)-m(\pi-\theta_1) = \frac{3\pi}{4}-m\pi$ by the symmetry $a_{16}^+(\pi-\theta)=-a_{16}^+(\theta)$.) Then the bounds $-\frac{\pi}{2} \le a_{16}^+(\theta_1) \le \frac{\pi}{2}$ (from Lemma \ref{lem:argument-even}) imply that 
    \begin{align} \label{eqn:temp.4.1}
        \frac{\pi}{4m} \le \theta_1 &= \frac{1}{m} \lrp{\frac{3\pi}{4} + a_{16}^+(\theta_1)} \le \frac{1}{m} \frac{5\pi}{4} < \frac{\pi}{2}.
    \end{align}
    Hence for all $\theta \in [\theta_1,\pi-\theta_1]$, we obtain the error bound
    \begin{align}
        \frac{\sqrt 2}{2} \cdot r_{16}^+(\theta)
        &\ge \frac{\sqrt 2}{2}  \cdot r_{16}^+(\theta_1)  \qquad\qquad \text{by Lemma \ref{lemma:radius-increasing}}
        \\
        &= 
        \frac{\sqrt 2}{2} \sqrt{\frac12 \cosh \lrp{\pi \sin \theta_1} + \frac12 \cos\lrp{\pi \cos \theta_1}} \\
        &\ge 
        \frac{\sqrt 2}{2} \sqrt{\frac12 \lrp{1 + \frac{(\pi \sin \theta_1)^2}{2}} - \frac12 } \\
        &=
        \frac{\sqrt 2}{2} \frac{\pi \sin \theta_1}{2} \\
        &\ge \frac{\sqrt 2}{2} \theta_1 \qquad\qquad \text{by \eqref{eqn:temp.4.1}} \\
        &\ge  \frac{\sqrt 2}{2} \frac{\pi}{4m} \qquad\qquad \text{by \eqref{eqn:temp.4.1}} \\
        &> \lrabs{E_f^+(\theta)} \qquad\qquad \text{by Proposition \ref{prop:error<radius-even}}. \label{eqn:temp.4.3}
    \end{align}
    This means that $\im\lrb{e^{-im\theta} q_f^+\lrp{\frac{e^{i\theta}}{\sqrt{16}}}}$ will have the same sign as $\sin\lrp{a_{16}^+(\theta)-m\theta}$ whenever $\lrabs{\sin\lrp{a_{16}^+(\theta)-m\theta}} \ge \frac{\sqrt 2}{2}$.

    Finally, we conclude that $a_{16}^+(\theta)-m\theta$ changes continuously from $-\frac{3\pi}{4}$ to $\frac{3\pi}{4}-m\pi$ over $\theta \in [\theta_1,\pi-\theta_1]$, and so $\lrabs{\sin\lrp{a_{16}^+(\theta)-m \theta}} \ge \frac{\sqrt{2}}{2}$ (alternating in sign) $m$ times over the same interval. 
    Hence $\im\lrb{e^{-im\theta} q_f^+\lrp{\frac{e^{i\theta}}{\sqrt{16}}}}$ has $m-1$ sign changes over $\theta \in (\theta_1,\pi-\theta_1)$, yielding $m-1$ distinct zeros in $(0,\pi)$. Note this is actually more zeros than we needed. In fact, we have shown the desired result in this case for $\beta_{16}^+=0$.   

    \textbf{Case 7: $N=16,\ \epsilon(f)=-1$.\\}
    In this case, we need to show that $\re\lrb{e^{-im\theta} q_f^+\lrp{\frac{e^{i\theta}}{\sqrt{16}}}}$ has at least $m-2$ distinct zeros over $\theta \in (0,\pi)$. 

    Here, $a_{16}^+(\theta)-m \theta$ changes continuously from $-\frac{\pi}{2}$ to $\frac{\pi}{2}-m\pi$ over $\theta \in [0,\pi]$ (by Lemma \ref{lem:argument-even}), and so $\cos\lrp{a_{16}^+(\theta)-m \theta} = \pm 1$ (alternatingly) $m-1$ times over the same interval.  
    
    However, Proposition \ref{prop:error<radius-even} here only yields the bound $\lrabs{E_f^+(\theta)} < \frac{1}{2m}$.
    As in the previous case, let $\theta_1$ be such that $a_{16}^+(\theta_1)-m\theta_1 = -\frac{3\pi}{4}$. Then for all $\theta \in [\theta_1,\pi-\theta_1]$, we obtain from \eqref{eqn:temp.4.3} that
    \begin{align}
        r_{16}^+(\theta)
        &> \lrabs{E_f^+(\theta)}.
    \end{align}
    This means that $\re\lrb{e^{-im\theta} q_f^+\lrp{\frac{e^{i\theta}}{\sqrt{16}}}}$ will have the same sign as $\cos\lrp{a_{16}^+(\theta)-m\theta}$ whenever $\cos\lrp{a_{16}^+(\theta)-m\theta} = \pm 1$.

    Finally, we conclude that $a_{16}^+(\theta)-m\theta$ changes continuously from $-\frac{3\pi}{4}$ to $\frac{3\pi}{4}-m\pi$ over $\theta \in [\theta_1,\pi-\theta_1]$, and so $\cos\lrp{a_{16}^+(\theta)-m \theta} = \pm1$ (alternatingly) $m-1$ times over the same interval. 
    Hence $\re\lrb{e^{-im\theta} q_f^+\lrp{\frac{e^{i\theta}}{\sqrt{16}}}}$ has $m-2$ sign changes over $\theta \in (\theta_1,\pi-\theta_1)$, yielding $m-2$ distinct zeros in $(0,\pi)$, as desired.
\end{proof}

\section{Proof of Theorem \ref{thm:main-thm-odd}}
\label{sec:proof-main-thm-odd}

In this section, we prove Theorem \ref{thm:main-thm-odd}. The details are very similar to those of Theorem \ref{thm:main-thm-even}, but we give a full proof here for completeness. Recall (by the reformulation of Section \ref{sec:reformulation}) that it suffices to show that the imaginary/real part of $e^{-im\theta}q_f^-\lrp{ \frac{e^{i \theta}}{\sqrt N}}$ has at least $m-1-\delta_f^- - 2\beta_N^-$ distinct zeros for $\theta \in (0,\pi)$. 
\begin{proof} 
    Recall that the level $N=1$ case was already shown in \cite{Conrey2013}. So in this proof, we only consider $N \ge 2$.
    Additionally, observe that there are only finitely many pairs $(N,k)$ with $2 \log_2 N \le k < k_N^-$ (see Proposition \ref{prop:error<radius-odd}). Hence to prove the desired result for $k < k_N^+$, we just checked these finitely many cases by computer \cite{ross-code}. For rest of the proof, we will then assume that $k \ge k_N^-$, so that we can apply Proposition \ref{prop:error<radius-odd}.
    
    \textbf{Case 1: $N \ge 5,\ \epsilon(f)=-1$.\\}
    In this case, we need to show that $\im\lrb{e^{-im\theta} q_f^-\lrp{\frac{e^{i\theta}}{\sqrt N}}}$ has at least $m-2$ distinct zeros over $\theta \in (0,\pi)$. Writing $q_f^-\lrp{\frac{e^{i\theta}}{\sqrt N}}$ in terms of its approximation function, we obtain
    \begin{align}
        \im\lrb{e^{-im\theta} q_f^-\lrp{\frac{e^{i\theta}}{\sqrt N}}}
        &= \im\lrb{e^{-im\theta}  \sin\lrp{\frac{2\pi e^{i\theta}}{\sqrt N}} 
        + e^{-im\theta} E_f^-(\theta)
        } \\
        &= \sin \lrp{a_N^-(\theta)-m \theta} r_N^-(\theta) + \im\lrb{e^{-im\theta} E_f^-(\theta)}. \label{eqn:temp.12.1}
    \end{align}
    Then recall that $\lrabs{E_f^-(\theta)} < r_N^-(\theta)$ (by Proposition \ref{prop:error<radius-odd}). Hence \eqref{eqn:temp.12.1} means that 
    $\im\lrb{e^{-im\theta} q_f^-\lrp{\frac{e^{i\theta}}{\sqrt N}}}$ will have the same sign as $\sin\lrp{a_N^-(\theta)-m \theta}$
    whenever $\sin\lrp{a_N^-(\theta)-m \theta} = \pm 1$. 
    
    Now, $a_N^-(\theta)-m \theta$ changes continuously from $0$ to $\pi-m\pi$ over $\theta \in [0,\pi]$ (by Lemma \ref{lem:argument-odd}), and so $\sin\lrp{a_N^+(\theta)-m \theta} = \pm 1$ (alternatingly) $m-1$ times over the same interval. Hence $\im\lrb{e^{-im\theta} q_f^-\lrp{\frac{e^{i\theta}}{\sqrt N}}}$ has $m-2$ sign changes over $\theta \in (0,\pi)$, yielding $m-2$ distinct zeros in $(0,\pi)$, as desired.

    \textbf{Case 2: $N \ge 5,\ \epsilon(f)=+1$.\\}
    In this case, we need to show that $\re\lrb{e^{-im\theta} q_f^-\lrp{\frac{e^{i\theta}}{\sqrt N}}}$ has at least $m-1$ distinct zeros over $\theta \in (0,\pi)$. Just like before, we have
    \begin{align}
        \re\lrb{e^{-im\theta} q_f^-\lrp{\frac{e^{i\theta}}{\sqrt N}}}
        &= \re\lrb{e^{-im\theta}  \cos\lrp{\frac{2\pi e^{i\theta}}{\sqrt N}} 
        + e^{-im\theta} E_f^-(\theta)
        } \\
        &= \cos \lrp{a_N^-(\theta)-m \theta} r_N^-(\theta) + \re\lrb{e^{-im\theta} E_f^-(\theta)}. \label{eqn:temp.12.2}
    \end{align}
    Hence 
    $\re\lrb{e^{-im\theta} q_f^-\lrp{\frac{e^{i\theta}}{\sqrt N}}}$ will have the same sign as $\cos\lrp{a_N^-(\theta)-m \theta}$
    whenever $\cos\lrp{a_N^-(\theta)-m \theta} = \pm 1$. 

    Now, $a_N^-(\theta)-m \theta$ changes continuously from $0$ to $\pi-m\pi$ over $\theta \in [0,\pi]$ (by Lemma \ref{lem:argument-odd}), and so $\cos\lrp{a_N^-(\theta)-m \theta} = \pm 1$ (alternatingly) $m$ times over the same interval. Hence $\re\lrb{e^{-im\theta} q_f^-\lrp{\frac{e^{i\theta}}{\sqrt N}}}$ has $m-1$ sign changes over $\theta \in (0,\pi)$, yielding $m-1$ distinct zeros in $(0,\pi)$, as desired.
        
    \textbf{Case 3: $2 \le N \le 3,\ \epsilon(f)=-1$.\\}
    In this case, we need to show that $\im\lrb{e^{-im\theta} q_f^-\lrp{\frac{e^{i\theta}}{\sqrt N}}}$ has at least $m-4$ distinct zeros over $\theta \in (0,\pi)$.

    Here, $a_N^-(\theta)-m \theta$ changes continuously from $\pi$ to $(-m+4)\pi$ over $\theta \in [0,\pi]$ (by Lemma \ref{lem:argument-odd}), and so $\sin\lrp{a_N^-(\theta)-m \theta} = \pm 1$ (alternatingly) $m-3$ times over the same interval. Hence $\im\lrb{e^{-im\theta} q_f^-\lrp{\frac{e^{i\theta}}{\sqrt N}}}$ has $m-4$ sign changes over $\theta \in (0,\pi)$, yielding $m-4$ distinct zeros in $(0,\pi)$, as desired.

    \textbf{Case 4: $2 \le N \le 3,\ \epsilon(f)=+1$.\\}
    In this case, we need to show that $\re\lrb{e^{-im\theta} q_f^-\lrp{\frac{e^{i\theta}}{\sqrt N}}}$ has at least $m-3$ distinct zeros over $\theta \in (0,\pi)$.

    Here, $a_N^-(\theta)-m \theta$ changes continuously from $\pi$ to $(-m+4)\pi$ over $\theta \in [0,\pi]$ (by Lemma \ref{lem:argument-odd}), and so $\cos\lrp{a_N^-(\theta)-m \theta} = \pm 1$ (alternatingly) $m-2$ times over the same interval. Hence $\re\lrb{e^{-im\theta} q_f^-\lrp{\frac{e^{i\theta}}{\sqrt N}}}$ has $m-3$ sign changes over $\theta \in (0,\pi)$, yielding $m-3$ distinct zeros in $(0,\pi)$, as desired.

    \textbf{Case 5: $N=4$.\\}
    Note that here, $\epsilon(f)$ is necessarily equal to $-1$, by \cite{Zhang2014}.
    In this case, we need to show that $\im\lrb{e^{-im\theta} q_f^-\lrp{\frac{e^{i\theta}}{\sqrt 4}}}$ has at least $m-2$ distinct zeros over $\theta \in (0,\pi)$.

    Here, $a_{4}^-(\theta)-m\theta$ changes continuously from $-\frac{\pi}{2}$ to $\frac{3\pi}{2}-m\pi$ over $\theta \in [0,\pi]$ (by Lemma \ref{lem:argument-odd}), and so $\sin\lrp{a_{4}^-(\theta)-m \theta} = \pm 1$ (alternatingly) $m-1$ times over the same interval. 
    
    However, Proposition \ref{prop:error<radius-odd} here only yields the bound $\lrabs{E_f^-(\theta)} < \frac{1}{2m}$. 
    Let $\theta_1$ be such that $a_{4}^-(\theta_1)-m\theta_1 = -\frac{3\pi}{4}$. Then the bounds $-\frac{\pi}{2} \le a_{4}^-(\theta_1) \le \frac{3\pi}{2}$ (from Lemma \ref{lem:argument-odd}) imply that 
    \begin{align} \label{eqn:temp.14.1}
        \frac{\pi}{4m} \le \theta_1 &= \frac{1}{m} \lrp{\frac{3\pi}{4} + a_{4}^-(\theta_1)} \le \frac{1}{m} \frac{9\pi}{4} < \frac{\pi}{2}.
    \end{align}
    Hence for all $\theta \in [\theta_1,\pi-\theta_1]$, we obtain the error bound
    \begin{align}
        \frac{\sqrt 2}{2} \cdot r_{4}^-(\theta)
        &\ge \frac{\sqrt 2}{2}  \cdot r_{4}^-(\theta_1)  \qquad\qquad \text{by Lemma \ref{lemma:radius-increasing}}
        \\
        &= 
        \frac{\sqrt 2}{2} \sqrt{\frac12 \cosh \lrp{2\pi \sin \theta_1} - \frac12 \cos\lrp{2\pi \cos \theta_1}} \\
        &\ge 
        \frac{\sqrt 2}{2} \sqrt{\frac12 \lrp{1 + \frac{(2\pi \sin \theta_1)^2}{2}} - \frac12 } \\
        &=
        \frac{\sqrt 2}{2} \frac{2\pi \sin \theta_1}{2} \\
        &\ge \frac{\sqrt 2}{2} 2\theta_1 \qquad\qquad \text{by \eqref{eqn:temp.14.1}} \\
        &\ge  \frac{\sqrt 2}{2} \frac{\pi}{2m} \qquad\qquad \text{by \eqref{eqn:temp.14.1}} \\
        &> \lrabs{E_f^-(\theta)} \qquad\qquad \text{by Proposition \ref{prop:error<radius-odd}}. \label{eqn:temp.14.3}
    \end{align}
    This means that $\im\lrb{e^{-im\theta} q_f^-\lrp{\frac{e^{i\theta}}{\sqrt{4}}}}$ will have the same sign as $\sin\lrp{a_{4}^-(\theta)-m\theta}$ whenever
    \\ 
    $\lrabs{\sin\lrp{a_{4}^-(\theta)-m\theta}} \ge \frac{\sqrt 2}{2}$.

    Finally, we conclude that $a_{4}^-(\theta)-m\theta$ changes continuously from $-\frac{3\pi}{4}$ to $\frac{7\pi}{4}-m\pi$ over $\theta \in [\theta_1,\pi-\theta_1]$, and so $\lrabs{\sin\lrp{a_{4}^-(\theta)-m \theta}} \ge \frac{\sqrt{2}}{2}$ (alternating in sign) $m-1$ times over the same interval. 
    Hence $\im\lrb{e^{-im\theta} q_f^-\lrp{\frac{e^{i\theta}}{\sqrt{4}}}}$ has $m-2$ sign changes over $\theta \in (\theta_1,\pi-\theta_1)$, yielding $m-2$ distinct zeros in $(0,\pi)$, as desired.        
\end{proof}

\section*{Acknowledgements}
This research was supported by the National Security Agency MSP grant H98230-24-1-0033.


\bibliographystyle{plain}
\bibliography{reference.bib}

\end{document}